\documentclass[11pt,reqno]{amsart}

\usepackage{amssymb, amsmath}
\usepackage{hyperref}
\usepackage[nohug]{diagrams}
\diagramstyle[labelstyle=\scriptstyle]

\newtheorem{proposition}{Proposition}[section]
\newtheorem{lemma}[proposition]{Lemma}
\newtheorem{corollary}[proposition]{Corollary}
\newtheorem{theorem}[proposition]{Theorem}

\newtheorem{problem}{Problem}[]
\newtheorem{question}{Question}[]

\theoremstyle{definition}

\newtheorem{example}[proposition]{Example}

\newtheorem{remark}[proposition]{Remark}

\newcommand{\thlabel}[1]{\label{th:#1}}
\newcommand{\thref}[1]{Theorem~\ref{th:#1}}
\newcommand{\selabel}[1]{\label{se:#1}}
\newcommand{\seref}[1]{Section~\ref{se:#1}}
\newcommand{\lelabel}[1]{\label{le:#1}}
\newcommand{\leref}[1]{Lemma~\ref{le:#1}}
\newcommand{\prlabel}[1]{\label{pr:#1}}
\newcommand{\prref}[1]{Proposition~\ref{pr:#1}}
\newcommand{\colabel}[1]{\label{co:#1}}
\newcommand{\coref}[1]{Corollary~\ref{co:#1}}
\newcommand{\relabel}[1]{\label{re:#1}}
\newcommand{\reref}[1]{Remark~\ref{re:#1}}

\newcommand{\eqlabel}[1]{\label{eq:#1}}
\newcommand{\equref}[1]{(\ref{eq:#1})}

\newcommand{\mc}{\mathcal}

\newcommand{\ve}{\varepsilon}
\newcommand{\ho}{{\rm Hom}}
\newcommand{\sq}{\square}

\newcommand{\alg}{{\rm Alg}}
\newcommand{\coalg}{{\rm CoAlg}}
\newcommand{\bialg}{{\rm BiAlg}}
\newcommand{\halg}{{\rm HopfAlg}}
\newcommand{\shalg}{{\rm SHopfAlg}}

\newcommand{\analog}{\rule{1cm}{0.01mm}}

\begin{document}

\title{On epimorphisms and monomorphisms of Hopf algebras}
\author{Alexandru Chirv\u asitu}

\address{
University of California, Berkeley, 970 Evans Hall \#3480, Berkeley, CA, 94720-3840 USA
}
\email{chirvasitua@gmail.com}


\begin{abstract}

We provide examples of non-surjective epimorphisms $H\to K$ in the category of Hopf algebras over a field, even with the additional requirement that $K$ have bijective antipode, by showing that the universal map from a Hopf algebra to its enveloping Hopf algebra with bijective antipode is an epimorphism in $\halg$, although it is known that it need not be surjective. Dual results are obtained for the problem of whether monomorphisms in the category of Hopf algebras are necessarily injective. We also notice that these are automatically examples of non-faithfully flat and respectively non-faithfully coflat maps of Hopf algebras. 

\end{abstract}

\maketitle

\section*{Introduction}\selabel{0}

In this paper, we are concerned primarily with the problem of whether epimorphisms in the category $\halg$ of Hopf algebras over a field $k$ are surjective, and the dual question of whether monomorphisms are injective. This makes sense in any concrete category; in \cite{Re}, for example, the corresponding problem (on epimorphisms) is solved for some familiar categories, such as groups, Lie algebras, $C^*$ and von Neumann algebras, compact groups, locally compact groups, etc. To our knowledge, the problem has not been treated in the literature in the context of Hopf algebras. 

Aside from being interesting and natural in their own right, the two questions do play a part in certain technical results on Hopf algebras. In \cite{AD}, for example, a paper concerned with exact sequences of Hopf algebras, these problems arise naturally several times. In the dual pair \cite[Lemmas 1.1.6, 1.1.10]{AD} it is shown that certain conditions on a morphism of Hopf algebras are implied by injectivity, and imply that the morphism in question is a monomorphism in $\halg$ (and similarly for surjectivity). Also, in a remark after \cite[Prop. 1.2.3]{AD}, the authors observe that in a diagram of the form 
\[
\begin{diagram}
0          &\rTo              &\bullet           &\rTo                &\bullet              &\rTo            &\bullet            &\rTo        &0\\
           &                  &\dTo<{{\rm id}}   &                    &\dTo<{\theta}        &                &\dTo<{{\rm id}}    &            & \\
0          &\rTo              &\bullet           &\rTo                &\bullet              &\rTo            &\bullet            &\rTo        &0
\end{diagram}
\]   
where the rows are what in that paper are called exact sequences of Hopf algebras (\cite [Prop. 1.2.3]{AD}), $\theta$ is both a monomorphism and an epimorphism of Hopf algebras. The authors then mention as unknown whether in this case it follows that $\theta$ is an isomorphism, or, in general, whether epimorphisms (monomorphisms) of Hopf algebras are surjective (injective). In other words, this is a direct reference to our problem. It is, however, the only such reference we could find in the literature. 

A much more well-documented problem, on the other hand, is the one known as Kaplansky's first conjecture. Strictly speaking, the conjecture/problem has undergone several transformations since its appearance in \cite{Ka}. It initially asked whether all Hopf algebras are (left and right) free modules over their Hopf subalgebras. At the time, this was already known to be false: Oberst and Schneider had constructed a counterexample in \cite{OSch}. 

There are several positive results on the problem: it holds for instance if the coradical of the large Hopf algebra is contained in the small one by a result of Nichols (this also follows from \cite[Cor. 2.3]{Ra2}), or if the large algebra is pointed (\cite{Ra1}), or in the finite dimensional case by the now famous Nichols-Zoeller theorem (\cite[Theorem 3.1.5]{Mo}). 

In view of the general negative answer, it makes sense to weaken the requirements: \cite[Question 3.5.4]{Mo} asks whether Hopf algebras are always (left and right) faithfully flat over their Hopf subalgebras. Again, this holds in various particular cases (commutative, or cocommutative, or even when the large algebra has cocommutative coradical; we give some references below, in \seref{2}, after \prref{halg>coalg}). 

In the commutative case, the problem of faithful flatness arose in the theory of affine algebraic groups, for which we refer to \cite{DG, Wa}. Indeed, faithful flatness for commutative Hopf algebras (\cite[Th. 3.1]{Ta3}) is crucial in Takeuchi's purely algebraic proof in \cite{Ta3} of the one-to-one correspondence between normal closed subgroup schemes and quotient affine group schemes of an affine group scheme. See \cite[Th. 5.2]{Ta3}, and also \cite[Chapters 13-16]{Wa} for an exposition of these results.     

Despite all of these positive partial results, in general, Hopf algebras are {\it not} faithfully flat over Hopf subalgebras (\cite[Remark 2.6, Cor. 2.8]{Sc}). At the end of \cite[$\S$2]{Sc}, Schauenburg asks what we refer to from now on as being the current version of Kaplansky's question (or problem): 

Are Hopf algebras with bijective antipode (left and right) faithfully flat over Hopf subalgebras with bijective antipode? 

Our interest in the question of faithful (co)flatness for Hopf algebras stems from the fact that there are strong connections between it and the problem of whether epimorphisms are surjective. These are understood by first noticing that epimorphisms of Hopf algebras can already be recognized at the level of algebras (\prref{halg>alg}) through an adjunction, and then that a faithfully flat epimorphism of algebras is an isomorphism (a well-known result, which we prove however, for the sake of completeness, in \prref{ff}). 

It follows that whenever we have non-surjective epimorphisms, we automatically have counterexamples to Kaplansky's question. In particular, our counterexamples to epi $\Rightarrow$ surjective in \seref{2} and \seref{3} recover those in \cite{Sc} for Kaplansky's problem, from this new point of view. On the other hand, it follows that epimorphisms {\it are} surjective when the conjecture holds (as mentioned above, for commutative or cocommutative, or pointed Hopf algebras, for instance). In the commutative case, for example, the fact that epi implies surjecivity can be translated into geometric language as follows (see \cite[Th. 5.2, (i)]{Ta3}; we are using the same notations as Takeuchi): 

A morphism ${\rm Sp}(H)\to{\rm Sp}(K)$ of affine groups is a monomorphism if and only if the corresponding Hopf algebra map $K\to H$ is surjective. 

Indeed, the category of commutative Hopf algebras is the opposite of that of affine groups, so a monomorphism in the latter is the same as an epimorphism in the former.

The paper is organized as follows:

In \seref{1} we introduce the notations and conventions to be used throughout. We also very briefly recall two characterizations of monomorphisms of coalgebras. 

\seref{2} is devoted to the questions asked above, in  precisely that form. They are quickly settled in the negative by the simple observation that the antipode of a Hopf algebra $H$, regarded as a Hopf algebra map from $H$ to $H^{op,cop}$ ($H$ with the opposite multiplication and coopposite comultiplication) is both a monomorphism and an epimorphism in $\halg$. We also need the facts, known for some time, that there are Hopf algebras with non-surjective (\cite{Ni}) or non-injective (\cite{Ta2, Sc}) antipode. 

In this same section, we highlight the interactions between the Kaplansky conjecture and the problem of whether epimorphisms in $\halg$ (the category of Hopf algebras) are surjective, as discussed above. We also look briefly at the dual situation: the problem of whether surjective Hopf algebra maps are faithfully coflat is linked to that of whether monomorphisms of Hopf algebras are injective through \prref{halg>coalg} and \prref{ff'}. 

Finally, as an interesting consequence of this discussion, we show in \prref{Scorad} that the antipode of a Hopf algebra is surjective whenever its image contains the coradical. 

In \seref{3} we modify our question by imposing stronger hypotheses (akin to what is done in \cite{Sc} for the Kaplansky problem): we ask whether an epimorphic inclusion of Hopf algebras must be surjective if the larger Hopf algebra has bijective antipode, as well as the dual question. Again, we prove that there are counterexamples (\coref{counterex}). These are obtained through two adjunctions between the categories of Hopf algebras and of Hopf algebras with bijective antipode. One is the adjunction constructed in \cite{Sc}, where it is shown that there is a free Hopf algebra with bijective antipode (denoted here by $K^*(H)$) on every hopf algebra $H$. We prove that the universal map $H\to K^*(H)$ is {\it always} an epimorphism of Hopf algebras, thus finding our counterexamples whenever it is not surjective (and this does occur). 

The other adjunction we use is the ``dual'' of the previous one: we prove that there is a cofree Hopf algebra $K_*(H)$ with bijective antipode on every Hopf algebra $H$, and that the universal map $K_*(H)\to H$ is always a monomorphism of Hopf algebras. Again, this provides us with counterexamples to mono $\Rightarrow$ injective whenever such a universal map is not injective. 

Because we find the analogy interesting, we carry out a parallel discussion for two adjunctions between the categories of bialgebras and Hopf algebras: there exist both a free and a cofree Hopf algebra on a bialgebra $B$ (the former follows from \cite{Ta1} and is constructed explicitly in \cite{Pa}; the existence of the latter is proven in \cite{Ag1}, and we construct it here). We denote these by $H^*(B)$ and $H_*(B)$ respectively. As before, we show that the unit of the first adjunction provides us with epimorphisms $B\to H^*(B)$ of bialgebras, and the counit of the other adjunction gives us monomorphisms $H_*(B)\to B$ of bialgebras. See \thref{adjepimono}. 

In \seref{4} we finish with some problems for the reader. 

First, there are the questions parallel to Kaplansky's conjecture in its current form and its dual: we would like to know whether epimorphisms (monomorphisms) of Hopf algebras are surjective (injective) when all Hopf algebras in question have bijective antipode. 

Secondly, we ask for necessary and sufficient conditions on a bialgebra in order that it be a quotient or a subbialgebra of a Hopf algebra, and also for necessary and sufficient conditions on a Hopf algebra in order that it be a quotient of one with bijective antipode. These are motivated by the result (which is an immediate consequence of \cite[Prop. 2.7]{Sc}) that a Hopf algebra $H$ is a Hopf subalgebra of one with bijective antipode iff its antipode $S_H$ is injective.   

\section{Preliminaries}\selabel{1}

Throughout this paper, $k$ will be an arbitrary field. Unless explicitly specified otherwise, homomorphisms, tensor products, algebras, coalgebras, and so on are over $k$. We work with several categories: $\alg$, $\coalg$, $\bialg$ and $\halg$ denote the categories of $k$-algebras, coalgebras, bialgebras and Hopf algebras, respectively. $\shalg$ stands for the category of Hopf algebras with a bijective antipode (the $S$ in front is supposed to remind the reader of the usual notation $S$ for the antipode of a Hopf algebra). If $x,y$ are objects in a category $\mc C$, we use the notation $\mc C(x,y)$ for the set of morphisms from $x$ to $y$ in $\mc C$. 

We use standard notations for opposite and coopposite structures: $A^{op}$ is the opposite of the algebra $A$, and $C^{cop}$ is the coopposite of the coalgebra $C$. 

For an algebra $A,\ _A\mc M$ denotes the category of left $A$-modules, and similarly, $\mc M_A$ is the category of right $A$-modules. For a coalgebra $C,\ ^C\mc M$ and $\mc M^C$ are the categories of left and, respectively, right $C$-comodules. 

For basic notions of category theory such as limits, colimits, adjunctions, comma categories and so on, we refer mainly to \cite{MacL}, but what we need can be found in most sources. Another example is \cite[Appendix]{Pa}. We use the language and notations in \cite{MacL}. At some point we do make use of the notion of locally presentable category, but only in passing. Everything we need on the subject can be found for instance in \cite[Chapter 1]{ARo}.  

For the structure maps of our objects we reserve the usual notation: $\eta,\Delta,\ve,S,\bar S$ denote, respectively, the unit, comultiplication, counit, antipode, and skew antipode of an appropriate object (algebra, Hopf algebra, etc.). We sometimes use subscripts to indicate the object in question: $S_H$ is the antipode of the Hopf algebra $H$, for instance. For a coalgebra $C$ and an algebra $A$, we regard $\ho(C,A)$ as an algebra in the usual way, under the convolution $*$; in Sweedler sigma notation (\cite[1.4.2]{Mo}; we have omitted the summation symbol), we have:
\[
(f*g)(c)=f(c_{(1)})g(c_{(2)}).
\]   
Recall that when $H$ is a Hopf algebra with antipode $S,\ A$ is an algebra, and $f\in\alg(H,A)$, the composition $fS$ is the inverse of $f$ with respect to the convolution operation $*$. Similarly, $Sf$ is the inverse of $f\in\coalg(C,H)$ for a coalgebra $C$ (\cite[Chapter IV, Lemma 4.0.3]{Sw}).  

We also require the notion of {\it faithful coflatness} over a coalgebra. The main definitions and properties regarding (faithful) coflatness can be found in \cite[Chapters 21]{BW}. Here, the notion replacing the tensor product is that of cotensor product over a coalgebra, for which we refer to \cite[Appendix 2]{Ta4} or \cite[Chapters 21,22]{BW}. 

We recall here a result on monomorphisms in $\coalg$. For a proof (of our lemma and the converses to its two statements), the reader can consult for example \cite{NT}, where quite a few characterizations of monomorphisms of coalgebras can be found; for even more such characterizations see \cite[T. 2.1]{Ag2}. As is customary in the literature, we denote by $\sq_D$ the cotensor product over the coalgebra $D$.

\begin{lemma}\lelabel{monocoalg}

Let $f:C\to D$ be a monomorphism in $\coalg$. Then the scalar coresriction $\mc M^C\to\mc M^D$ is full, and the comultiplication $\Delta_C$ is a bijection of $C$ onto $C\sq_DC\subseteq C\otimes C$. 

\end{lemma} 

\section{First version of the problem}\selabel{2}

The most general form of the problem we are concerned with in this paper consists of the two analogous questions of whether epi(mono)morphisms in the category $\halg$ are surjective (resp. injective). Notice that a map of Hopf algebras $f:H\to K$ is an epimorphism iff the inclusion of the image of $H$ in $K$ is epi. Similarly, when we investigate monomorphisms, we can assume that they are surjective. We will sometimes do this without mentioning it explicitly. 

We shall see that the answers to the two questions are negative, using the following simple observation:

\begin{proposition}\prlabel{Sepi}

The antipode $S$ of a Hopf algebra $H$ is both an epimorphism and a monomorphism in $\halg$ from $H$ to $H^{op,cop}$. 

\end{proposition}

\begin{proof}

$S$ is an epimorphism iff for any Hopf algebra $K$, the map 
\[
\halg(H^{op,cop},K)\to \halg(H,K)
\] 
induced by it and defined by $f\mapsto fS$ is injective. More generally, if $A$ is an algebra and $f$ is an algebra map from $H^{op}$ to $A$, then $fS$ is the inverse of $f$ in the monoid $\ho(H^{cop},A)$ under convolution (here, $H$ is viewed only as a coalgebra). It follows that $f$ is uniquely determined by $fS$, which is what we needed. 

The statement that $S$ is mono is proven similarly: we have to show that for any Hopf algebra $K$, the map 
\[
\halg(K,H)\to \halg(K,H^{op,cop}) 
\]
given by $f\mapsto Sf$ is injective. Again, this holds more generally, if we replace $K$ with a coalgebra $C$ and $\halg$ with $\coalg$, simply by noticing that $Sf$ is the inverse of $f\in\coalg(C,H)$ in $\ho(C,H)$.
\end{proof}

The negative answers to our two questions now follow from the fact that there exist Hopf algebras with pathological (non-surjective or non-injective) antipode. A Hopf algebra with non-bijective antipode is already constructed in \cite{Ta1}. However, we need the more specific result (\cite{Ni}) that Takeuchi's algebra has a non-surjective antipode. In fact, Nichols also shows in \cite{Ni} that the antipode is injective. The Hopf algebra in question is the free Hopf algebra $H(M_n(k)^*)$ (a construction introduced in \cite{Ta1}) on the coalgebra $M_n(k)^*$, the dual of the matrix algebra $M_n(k)$ for $n>1$. We shall have more to say about such universal constructions in the next section. 

As for the injectivity of the antipode, Takeuchi proves (\cite[Theorem 9]{Ta2}) that either the same free Hopf algebra $H(M_n(k)^*)$ has a non-injective antipode (as mentioned above, we know this to be false from \cite{Ni}), or some quotient of $H(M_{2n}(k)^*)$ does. Also, Schauenburg constructs in \cite{Sc} a Hopf algebra with a surjective, non-injective antipode. Given these pathological examples and the previous proposition, we get

\begin{corollary}

There exist (injective) non-surjective epimorphisms in $\halg$, as well as (surjective) non-injective monomorphisms. 

\end{corollary}

In the next section we will also see examples of non-surjective epimorphisms $H\to K$ with $K$ having a bijective antipode, and of non-injective monomorphisms $H\to K$ with $H$ having a bijective antipode. We do not know if both algebras can be chosen to have bijective antipode in such counterexamples. 

As it turns out, the problem epi vs. surjective is linked to Kaplansky's first conjecture. The more modern version of this conjecture asked whether all Hopf algebras are (left and right) faithfully flat over their Hopf subalgebras (\cite[Question 3.5.4]{Mo}). Schauenburg gave some counterexamples in \cite{Sc}, and strengthened the hypotheses further: are Hopf algebras with bijective antipode faithfully flat over Hopf subalgebras with bijective antipode? In order to see the connection between the two problems, we need the following simple result on faithful flatness:

\begin{proposition}\prlabel{ff}

Let $\iota:A\to B$ be a left faithfully flat extension of algebras. If $\iota$ is an epimorphism in $\alg$, then it is an isomorphism.  

\end{proposition}

\begin{proof}

The fact that $\iota$ is epi implies that $b\otimes_A 1=1\otimes_Ab$ in $B\otimes_AB$ for all $b\in B$ (\cite[Chapter XI, Prop. 1.1]{St}). It follows immediately from this last condition that the map $\iota\otimes_AI_B:B\to B\otimes_AB$ is an isomorphism of right $B$-modules (actually, it follows that the map is surjective; the injectivity is clear from the fact that the multiplication $B\otimes_AB\to B$ is a left inverse for $\iota\otimes_AI_B$). By faithful flatness, $\iota$ must be an isomorphism of right $A$-modules. 
\end{proof}

The fact that the forgetful functor $\halg\to\alg$ has a right adjoint (\cite[Theorem 3.3]{Ag1}; the result is dual to Takeuchi's construction of a free Hopf algebra on a coalgebra in \cite{Ta1}), together with the easy-to-prove results that (a) left adjoints preserve epimorphisms and (b) faithful functors reflect epimorphisms, imply

\begin{proposition}\prlabel{halg>alg}

A morphism of Hopf algebras $f:H\to K$ is an epimorphism if and only if it is an epimorphism in $\alg$, when viewed as a map of algebras. 

\end{proposition}

We also record the dual statement, which follows by the dual argument: by \cite{Ta1} the forgetful functor $\halg\to\coalg$ is a right adjoint, and hence preserves monomorphisms.

\begin{proposition}\prlabel{halg>coalg}

A morphism of Hopf algebras $f:H\to K$ is a monomorphism if and only if it is a monomorphism in $\coalg$, when viewed as a map of coalgebras. 

\end{proposition}

\prref{ff} and \prref{halg>alg} show that epimorphisms of Hopf algebras are indeed surjective whenever Kaplansky's conjecture holds, i.e. in those stuations when we do have faithful flatness. Such situations are, for instance, the case when (same notations as in the statement of \prref{halg>alg}) $K$ is commutative, or has cocommutative coradical, or is pointed (\cite[Theorem 3.1]{Ta3} takes care of the cases when $K$ is either commutative or cocommutative, but \cite[Theorem 3.2]{Ta3} easily implies the cocommutative coradical and the pointed cases as well; later, Radford proved in \cite{Ra1} that pointed Hopf algebras are, in fact, free over their Hopf subalgebras). 

The contrapositive is that counterexamples to epi $\Rightarrow$ surjective are counterexamples to Kaplansky's first conjecture. In particular, by \prref{Sepi}, we recover Schauenburg's example (\cite[Remark 2.6]{Sc}) $S(H)\subset H$ of a non-faithfully flat inclusion of Hopf algebras whenever the antipode $S$ of $H$ is not surjective. 

The fact that epi implies surjectivity in the cocommutative case, for example, can be used, together with some adjunctions, to prove the classical results that epimorphisms are surjective in the categories of groups or Lie algebras. See also \cite[Prop. 3,4]{Re} for an interesting method of proof, using split extensions of groups and Lie algebras, respectively.

The discussion above on the connection between faithful flatness over Hopf subalgebras and epimorphisms in $\halg$ can be dualized: one can ask when a surjection of Hopf algebras is faithfully coflat (see \seref{1}), and investigate the relation between this question and the problem of determining if/when monomorphisms of Hopf algebras are injective. Faithful coflatness appears in \cite{AD}, for example, along with faithful flatness, as an important technical condition (see the dual pair of results \cite[Corollaries 1.2.5, 1.2.14]{AD}). 

We now want to prove the dual of \prref{ff}. Together with \prref{halg>coalg}, it will establish the connection between faithful coflatness and the injectivity of monomorphisms in $\halg$: if the surjective monomorphism $H\to K$ happens to be faithfully coflat, then it is an isomorphism. Again, the contrapositive is that whenever we have a non-injective monomorphism in $\halg$ (which we may as well assume is surjective), we have an example of non-faithfully coflat surjection of Hopf algebras.

\begin{proposition}\prlabel{ff'}

Let $f:C\to D$ be map of coalgebras, making $C$ left faithfully coflat over $D$. If $f$ is a monomorphism in $\coalg$, then it is an isomorphism. 

\end{proposition}

\begin{proof}

Since $f$ is a monomorphism, we know from \leref{monocoalg} that the canonical map $C\to C\sq_DC$ is bijective. The map 
\[
f\sq_DI_C:C\sq_DC\to D\sq_DC\cong C
\]
is a left inverse for $C\to C\sq_DC$, so it must also be bijective. Faithful coflatness implies that $-\sq_DC$ reflects isomorphisms, so $f$ must be an isomorphism. 
\end{proof}

Finally, we end this section with a consequence of \prref{Sepi} giving a sufficient condition for the antipode of a Hopf algebra to be surjective. We do not use this result elsewhere in the paper.

\begin{proposition}\prlabel{Scorad}

Let $H$ be a Hopf algebra with antipode $S$. If $S(H)$ contains the coradical $H_0$ of $H$, then $S$ is surjective. 

\end{proposition}

\begin{proof}

\prref{Sepi} says that the inclusion $S(H)\to H$ is epi. On the other hand, as $S(H)$ contains the coradical $H_0$, the inclusion is faithfully flat (in fact, $H$ is even free over $S(H)$, by a result of Nichols; it is also an immediate consequence of \cite[Cor. 2.3]{Ra2}). By \prref{ff}, we are done: the inclusion of $S(H)$ in $H$ must be surjective. 
\end{proof}

\section{Adjunctions and bijective antipodes}\selabel{3}

We have seen in the previous section that one can find both non-surjective epimorphisms and non-injective monomorphisms in the category $\halg$. We now strengthen the hypotheses: for epimorphisms $H\to K$, we ask that $K$ have bijective antipode. Similarly, for monomorphisms $H\to K$, we ask that $H$ have bijective antipode. Again, we find counterexamples in these situations. I do not know what happens if {\it both} Hopf algebras are required to have bijective antipodes. 

The construction is as follows: 

In \cite{Sc}, Schauenburg constructs the left adjoint, which we denote here by $K^*$, of the inclusion $i:\shalg\to\halg$ (recall that $\shalg$ is the category of Hopf algebras with bijective antipode; we will sometimes omit the inclusion functor), and proves (\cite[Cor. 2.8]{Sc}) that the unit $H\to K^*(H)$ of the adjunction is a non-faithfully flat inclusion of Hopf algebras whenever $H$ has injective non-bijective antipode (in fact, he proves more, namely that the inclusion does not have a certain property (P), weaker that faithful flatness). We show here that the unit $H\to K^*(H)$ is always an epimorphism of Hopf algebras. We also prove that the inclusion $i$ has a right adjoint $K_*$, and that the counit $K_*(H)\to H$ of the resulting adjunction is always a monomorphism of Hopf algebras. These will be examples of non-surjective epimorphisms and non-injective monomorphisms, with our extra requirements on the antipodes, when the antipode of Hopf algebra $H$ is ``pathological''. 

There seems to be an interesting parallel between the pairs of categories $\bialg,\halg$ on the one hand and $\halg,\shalg$ on the other; in order to emphasize it, we also carry out the arguments outlined above for the inclusion $j:\halg\to\bialg$. The existence of the left adjoint to this inclusion is a classical result of Takeuchi (\cite{Ta1}; even though Takeuchi passes directly from coalgebras to Hopf algebras, the intermediary adjoint from $\halg$ to $\bialg$ is easily deduced, and the construction is given explicitly in \cite[Theorem 2.6.3]{Pa}), and the existence of a right adjoint is proven in \cite[Theorem 3.3]{Ag1}. We state here the existence result for these adjoints:

\begin{theorem}\thlabel{exadj}

(a) The inclusion $j:\halg\to\bialg$ has both a left adjoint $H^*$ and a right adjoint $H_*$. 

(b) The inclusion $i:\shalg\to\halg$ has both a left adjoint $K^*$ and a right adjoint $K_*$.  
 
\end{theorem}

Before going into the proof (which will consist mainly of the constructions of the right adjoints to the inclusions, since the left adjoints are constructed explicitly in \cite{Ta1,Pa} and \cite{Sc} as indicated above), we state and prove the main result of this section, and derive some consequences. We keep the notations from the statement of \thref{exadj}.

\begin{theorem}\thlabel{adjepimono}

(a) For every bialgebra $B$, the component $B\to H^*(B)$ of the unit of the adjunction $(H^*,j)$ is an epimorphism of bialgebras, and the component $H_*(B)\to B$ of the counit of the adjunction $(j,H_*)$ is a monomorphism of bialgebras. 

(b) For any Hopf algebra $H$, the unit $H\to K^*(H)$ of the adjunction $(K^*,i)$ is an epimorphism of Hopf algebras, and the counit $K_*(H)\to H$ of the adjunction $(i,K_*)$ is a monomorphism of Hopf algebras. 

\end{theorem}

For the proofs we require a category-theoretic lemma, which we state after some notations.

Let $\mc C,\mc D$ be two categories, and $U:\mc C\to\mc D$ a functor with a left adjoint $F$ and a right adjoint $G$. Denote by $\alpha:I_{\mc D}\to UF$ and $\beta:UG\to I_{\mc D}$ the unit of the adjunction $(F,U)$ and the counit of the adjunction $(U,G)$, respectively. We then have:

\begin{lemma}\lelabel{epi>mono}

With the notations above, $\alpha_d:d\to UF(d)$ is an epimorphism for every object $d\in\mc D$ iff $\beta_d:UG(d)\to d$ is a monomorphism for every object $d\in\mc D$.  

\end{lemma}

\begin{proof}

For each pair of objects $d,d'\in\mc D$, we have a commutative diagram

\[
\begin{diagram}
\mc D(UF(d),d')       &             &           \\
\dTo                  &\rdTo        &           \\
\mc C(F(d),G(d'))     &             &\mc D(d,d')\\
\dTo                  &\ruTo        &           \\
\mc D(d,UG(d'))       &             &   
\end{diagram}
\]
where the two vertical arrows are the bijections given by the two adjunctions, and the two diagonal arrows are induced by $\alpha_d$ (the upper arrow) and $\beta_d$ (the lower arrow). 

The fact that $\alpha_d$ is an epimorphism for all $d$ is equivalent to the upper diagonal arrow being an injection for all pairs $d,d'$. Similarly for $\beta_d$ and the lower diagonal arrow. But since the vertical maps are bijections, the conditions that the upper and respectively lower diagonal arrow be an injection for all pairs $d,d'$ are equivalent. 
\end{proof}

\renewcommand{\proofname}{Proof of \thref{adjepimono}}
\begin{proof}

By applying \leref{epi>mono} to the two situations depicted in (a) and (b) (with the functor $U$ being the inclusion $j$ and $i$ respectively), we conclude that it suffices to prove one of the two statements in each of (a) and (b). It is enough, for instance, to show that the units of the two adjunctions $(H^*,j)$ and $(K^*,i)$ are epimorphisms. 

(a) We want to show that $\alpha:B\to H^*(B)$ is an epimorphism in $\bialg$ (strictly speaking, it should be $jH^*(B)$). Let $S$ be the antipode of $H^*(B)$. The subalgebra $H$ of $H^*(B)$ generated by $S^n(\alpha(B)),\ n\ge 0$ is a Hopf subalgebra: it is an algebra by definition, it is closed under $S$ again by definition, and it's a subcoalgebra because all the $S^n(\alpha(B))$ are. This means that $B\to H$ is a subobject of the initial object $B\to H^*(B)$ in the comma category $B\downarrow\halg$ (\cite[II$\S$6]{MacL}), and hence that $H=H^*(B)$.

Now consider a morphism of bialgebras $f:H^*(B)\to B'$. Then $fS\alpha$ is the inverse of $f\alpha$ in $\ho(B,B')$ under convolution, $fS^2\alpha$ is the inverse of $fS\alpha$ in $\ho(B^{cop},B')$, and so on. Because, as we have just seen, $H^*(B)$ is generated as an algebra by the iterations of $\alpha(B)$ under $S$, $f$ is uniquely determined by $f\alpha$. This is precisely the condition required in order that $\alpha$ be an epimorphism of bialgebras.  

(b) The proof runs parallel to that from (a): instead of the antipode, we now use the inverse $\bar S$ of the antipode $S$ of $K^*(H)$. Again, let $K$ be the subalgebra of $K^*(H)$ generated by $\bar S^n(\alpha(H)),\ n\ge 0$. Arguing as before, we conclude that $K=K^*(H)$, i.e. that $K^*(H)$ is generated as an algebra by the images of $\alpha(H)$ through the iterations of $\bar S$, and hence that a Hopf algebra map $f:K^*(H)\to H'$ is uniquely determined by $f\alpha:H\to H'$. 
\end{proof}
\renewcommand{\proofname}{Proof}

As a consequence, we have:

\begin{corollary}\colabel{counterex}

(a) If $B$ is a sub-bialgebra of a Hopf algebra such that $B$ itself is not Hopf, then $B\to H^*(B)$ is an injective, non-surjective epimorphism of bialgebras. Similarly, if the bialgebra $B$ is not Hopf but is a quotient of a Hopf algebra, then $H_*(B)\to B$ is a surjective, non-injective monomorphism of bialgebras. 

(b) If $H$ does not have bijective antipode but is contained in a Hopf algebra with bijective antipode, then $H\to K^*(H)$ is an injective, non-surjective epimorphism of Hopf algebras. Similarly, if $H$ does not have bijective antipode but is a quotient of a Hopf algebra with bijective antipode, then $K_*(H)\to H$ is a non-injective, surjective monomorphism. 

\end{corollary}

\begin{proof}

(a) Since the inclusion of $B$ in a Hopf algebra factors through $B\to H^*(B)$, the latter must be an injective map. The rest follows immediately from \thref{adjepimono}. For the second statement the dual argument works. 

(b) is entirely analogous to (a). 
\end{proof}

Examples as in the previous corollary actually exist. Focusing on (b), the Hopf algebra case, such examples can be found in \cite{Sc}: {\it any} Hopf algebra $H$ with injective non-bijective antipode (such as the free Hopf algebra on the coalgebra $M_n(k)^*,\ n\ge 2$, according to \cite{Ni}) injects properly into $K^*(H)$, and also, an example is given of a Hopf algebra with non-bijective antipode which is a quotient of a Hopf algebra with bijective antipode: it is a quotient of the free Hopf algebra with bijective antipode on the coalgebra $M_4(k)^*$. In conclusion, we have:

\begin{corollary}\colabel{main}

There is an epimorphic inclusion $H\to K$ of Hopf algebras with $K$ having a bijective antipode. Similarly, there is a monomorphic surjection $H\to K$ of Hopf algebras with $H$ having a bijective antipode. 

\end{corollary}

Next, we give explicit constructions for the right adjoints to the inclusions $j:\halg\to\bialg$ and $i:\shalg\to\halg$. In particular, this solves \cite[Problem 2]{Ag1}, which asks for a construction for the right adjoint to $j$, shown there to exist by the Special Adjoint Functor Theorem (the dual of \cite[V$\S$8, Corollary]{MacL}). 

Throughout, we shall make free use of the fact that the following categories are all complete and cocomplete: $\alg,\coalg,\bialg,\halg$. In fact, they are locally presentable, and locally presentable categories are cocomplete (by definition: \cite[Def. 1.17]{ARo}) and complete (\cite[Remark 1.56]{ARo}). 

The local presentability is proven up to bialgebras in \cite{Po1} in the more general setting of monoids, comonoids and bimonoids in a symmetric monoidal category with some extra assumptions, which are all satisfied by the category of $k$-vector spaces (see Summary 4.3 in that paper); that $\halg$ is locally presentable follows from \cite[Prop. 4.3]{Po2} and the fact that by \cite{Ta1}, the forgetful functor $\halg\to\coalg$ has a left adjoint (this is the argument used in the proof of \cite[Theorem 2.6]{Ag1}). Alternatively, one could prove the local presentability of these categories directly, but we do not go into these details here.

We start the construction of adjoints with the inclusion $j:\halg\to\bialg$.

\renewcommand{\proofname}{Proof of \thref{exadj} (a)}
\begin{proof}

As mentioned before, we only construct the right adjoints, since explicit constructions for the left adjoints can be found in the literature, in the sources cited above.

We simply dualize the construction from \cite[Theorem 2.6.3]{Pa}. As that proof is very detailed, and most arguments here are simply dualizations of those, we will only indicate how the construction goes, leaving out simple verifications.   

Let $B$ be a bialgebra, and let $P$ be the product (in the category $\bialg$) of the bialgebras $B_n,\ n\ge 0$, where $B_n=B$ if $n$ is even, and $B_n=B^{op,cop}$ if $n$ is odd. Denote by $\pi_n$ the structure maps $P\to B_n$ of the product of bialgebras, and let $\eta, \ve$ be the unit and counit of $P$ respectively. By the universality of the product, there is a unique bialgebra map $S$ such that 
\[
\begin{diagram}
P^{op,cop}                &\rTo^S          &P   \\
\dTo<{\pi_{n+1}}          &                &\dTo>{\pi_n} \\
B_{n+1}^{op,cop}          &\rTo^{{\rm id}} &B_n  
\end{diagram}
\] 
commutes for all $n\ge 0$. 

Let $H_*(B)$ be the sum of all subcoalgebras $C\subseteq P$ on which $S$ behaves like an antipode. Specifically, the condition such a coalgebra $C$ is supposed to satisfy is
\begin{equation}\eqlabel{antipode}
c_{(1)}S(c_{(2)})=\eta\ve(c)=S(c_{(1)})c_{(2)},\ \forall c\in C.
\end{equation}
As the notation suggests, this is the object we are looking for. $H_*(B)$ is by definition a subcoalgebra of $P$, and \equref{antipode} holds with $H_*(B)$ instead of $C$. In other words, $H_*(B)$ is the largest subcoalgebra of $P$ on which $S$ acts as an antipode. It is an easy matter now to prove that $H_*(B)$ is closed under multiplication and the action of $S$ (and it clearly contains the unit $1_P$), so it is, in fact, a Hopf subalgebra of $P$ with antipode $S$.   

We now want to prove that $\beta:H_*(B)\to B$, the composition of $\pi_0:P\to B$ with the inclusion $H_*(B)\to P$, is universal from a Hopf algebra to $B$. So let $f:H\to B$ be a bialgebra map from a Hopf algebra $H$ to $B$. The maps 
\[
f_n=f\circ S_H^n:H\to B_n,\ n\ge 0
\]
are bialgebra morphisms, and so define a bialgebra map $\tilde f:H\to P$ with $\pi_0\circ \tilde f=f$. First, we want to show that $\tilde f$ intertwines $S$ and $S_H$:
\[
\begin{diagram}
H                  &\rTo^{S_H}        &H              \\
\dTo<{\tilde f}    &                  &\dTo>{\tilde f}\\
P                  &\rTo^S            &P
\end{diagram}
\]
In turn, this follows from the universality of the product $P$ if we show that
\begin{equation}\eqlabel{0}
\pi_n\circ\tilde f\circ S_H=\pi_n\circ S\circ\tilde f,\ \forall n\ge 0
\end{equation}
as maps from $H$ to $B$. On the one hand, from the definition of $\tilde f$, we get
\begin{equation}\eqlabel{1}
\pi_n\circ\tilde f\circ S_H=f_n\circ S_H=f\circ S_H^{n+1}=\pi_{n+1}\circ\tilde f,
\end{equation}
and on the other hand, from the definition of $S$, we have 
\begin{equation}\eqlabel{2}
\pi_n\circ S\circ\tilde f=\pi_{n+1}\circ\tilde f,
\end{equation}
because $\pi_n\circ S=\pi_{n+1}$ as maps from $P$ to $B$ (we identify the underlying sets of all $B_n$). \equref{1} and \equref{2} now prove the desired equality \equref{0}. 

Because $\tilde f$ intertwines $S,S_H$ and $S_H$ is the antipode of $H$, it follows that $S$ is the antipode of $\tilde f(H)$. The definition of $H_*(B)$ now implies that the image of $\tilde f$ is contained in $H_*(B)$, i.e. $\tilde f$ factors through $H_*(B)\subseteq P$. In other words, we have just shown that any bialgebra map $f:H\to B$ factors as
\[
\begin{diagram}
H           &\rTo^{\tilde f}         &H_*(B)      \\
            &\rdTo>f                 &\dTo>\beta  \\
            &                        &B           
\end{diagram}
\]
It remains to prove that in such a diagram, $\tilde f$ is unique. Again, $\tilde f$ is determined by the sequence of maps $\pi_n\tilde f$ (also regarding $\pi_n$ as a map from $H_*(B)\subseteq P$ to $B_n$). But notice that, because $\tilde f$ commutes with the antipodes, we have
\[
\pi_n\circ\tilde f\circ S_H=\pi_n\circ S\circ\tilde f=\pi_{n+1}\circ \tilde f. 
\] 
This means that $\pi_{n+1}\tilde f$ is the inverse of $\pi_n\tilde f$ in $\ho(H,B_n)$ under convolution, and hence that the sequence $\pi_n\tilde f$ is uniquely determined by $\pi_0\tilde f=f$. This finishes the proof. 
\end{proof}
\renewcommand{\proofname}{Proof}

We now want to obtain the right adjoint to the inclusion $i:\shalg\to\halg$ as a direct consequence of \thref{exadj} (a) above. For this, we need

\begin{lemma}\lelabel{skant}

Let $B$ be a bialgebra with a skew antipode $\bar S_B$. Then, the cofree Hopf algebra $H_*(B)$ constructed above also has a skew antipode $\bar S$. Consequently, the antipode $S$ of $H_*(B)$ is bijective.   

\end{lemma}

\begin{proof}

The last statement follows immediately from the first, as it is well-known that a Hopf algebra has a skew antipode iff its antipode is bijective, in which case the skew antipode is the inverse of the antipode (\cite[Lemma 1.5.11]{Mo}). We focus on showing that the antipode $S$ of $H_*(B)$ is bijective. 

We use the notations from the proof of \thref{exadj} (a). Recall that there are maps $\pi_n$ from $H_*(B)$ to $B_n,\ n\ge 0$, where $B_n$ is $B$ for even $n$ and $B^{op,cop}$ for odd $n$. $\pi_0$ is universal, and the maps $\pi$ satisfy
\begin{equation}\eqlabel{skant}
\pi_nS=\pi_{n+1},\ \forall n\ge 0. 
\end{equation}

From the universality of $\pi_0:H_*(B)\to B$, we can find a unique Hopf algebra map $\bar S$ making the following diagram of bialgebra morphisms commutative:
\[
\begin{diagram}
H_*(B)^{op,cop}               &\rTo^{\bar S}                &H_*(B)         \\
\dTo<{\pi_0}                  &                             &\dTo<{\pi_0}   \\
B^{op,cop}                    &\rTo^{\bar S_B}              &B
\end{diagram}
\]
The aim is to show that $\bar S$ is a composition inverse to $S$. Complete this diagram to the left with another square (commutative by \equref{skant} for $n=0$):
\[
\begin{diagram}
H_*(B)              &\rTo^S                   &H_*(B)^{op,cop}               &\rTo^{\bar S}                &H_*(B)         \\
\dTo<{\pi_1}        &                         &\dTo<{\pi_0}                  &                             &\dTo<{\pi_0}   \\
B^{op,cop}          &\rTo^{{\rm id}}          &B^{op,cop}                    &\rTo^{\bar S_B}              &B
\end{diagram}
\]

Again by the universality of $\pi_0$, the composition $\bar SS$ is the unique Hopf algebra map making the outer rectangle commutative. If we prove that the identity on $H_*(B)$ also makes the outer rectangle commutative, we will have shown that $\bar S$ is a left composition inverse for $S$. In other words, we now want to show that
\begin{equation}\eqlabel{inverses}
\pi_0=\bar S_B\pi_1. 
\end{equation}

Since $\bar S_B$ is an antipode for $B^{cop}$ and $\pi_1$ is in $\coalg(H_*(B),B^{cop})$, the composition $\bar S_B\pi_1$ is the convolution inverse of $\pi_1$ in $\ho(H_*(B),B)$ (or $\ho(H_*(B),B^{cop})$, the algebra structure under convolution is the same). On the other hand, \equref{skant} with $n=0$ shows that $\pi_0$ is also the convolution inverse of $\pi_1$ in $\ho(H_*(B),B)$. This implies the desired equality \equref{inverses}. 

We have just shown that $\bar SS=I_{H_*(B)}$. Deducing now that $S\bar S$ is also the identity is easy: $S=S\bar SS$ is the convolution inverse of both $I_{H_*(B)}$ and of $S\bar S$ in ${\rm End}(H_*(B))$. 
\end{proof}

We now have what we need to finish the proof of \thref{exadj}.

\renewcommand{\proofname}{Proof of \thref{exadj} (b)}
\begin{proof}

Let $H$ be a Hopf algebra. The bialgebra $B=H^{op}$ has a skew antipode, namely $S_H$. According to \leref{skant}, the antipode of the cofree Hopf algebra $H_*(B)$ on $B$ is bijective. The universal bialgebra map 
\begin{equation}\eqlabel{exadj (b)}
\beta:H_*(H^{op})\to H^{op}
\end{equation}
induces a bialgebra map denoted by the same symbol:
\[
\beta:(H_*(H^{op}))^{op}\to H. 
\]
I claim that this is universal from a Hopf algebra with bijective antipode to $H$. In other words, we have
\[
K_*(H)=(H_*(H^{op}))^{op},
\]
with the obvious universal map $\beta$ to $H$. 

To see this, let $f:K\to H$ be a Hopf algebra map, with $K$ having bijective antipode. $f$ is then also a bialgebra morphism from the Hopf algebra $K^{op}$ to $H^{op}$, and hence factors uniquely through $\beta$ by the universality of \equref{exadj (b)}. This gives a unique map $\tilde f$, say, from $K^{op}$ to $H_*(H^{op})$. $\tilde f$ will then also be the unique Hopf algebra map from $K$ to $(H_*(H^{op}))^{op}$ through which $f$ factors, and the proof is finished. 
\end{proof}
\renewcommand{\proofname}{Proof}

\begin{remark}\relabel{altadj}

Although we prefer the construction used above because it shows how \thref{exadj} (b) follows directly from (a), there is more than one way of introducing the right adjoint to $i:\shalg\to\halg$. 

One idea, for instance, would be to dualize Schauenburg's construction from \cite[Prop. 2.7]{Sc}: $K_*(H)$ is the limit of the inverse system of Hopf algebras $u_n:H_{n+1}\to H_n,\ n\ge 0$, where all $H_n$ are $H$, and all $f_n$ are equal to the square $S_H^2$ of the antipode $S_H$. 

Alternatively, we could imitate the construction appearing in \thref{exadj} (a), by using a product of bialgebras $B_n$ indexed by the integers instead of the natural numbers, with $B_n=B$ for $n$ even and $B_n=B^{op,cop}$ for odd $n$ (just as before). 

This observation works the other way around too: the {\it left} adjoint of the inclusion $i:\shalg\to\halg$, denoted by $K^*$, can be constructed in the same manner, using the left adjoint of $j:\halg\to\bialg$ from \thref{exadj} (a). Just as in the previous proof, we have 
\[
K^*(H)=(H^*(H^{op}))^{op}. 
\]

\end{remark}

\section{Some comments and problems}\selabel{4}

As remarked several times before, I do not know whether counterexamples as in \coref{main} still exist if we require that {\it both} Hopf algebras $H$ and $K$ have bijective antipode. 

In the spirit of the connections we have noticed above between faithful flatness/coflatness and the problem of category-theoretic conditions (epimorphisms, monomorphisms) vs. set-theoretic conditions (surjectivity, injectivity), we ask:

\begin{question}

Is an epimorphism of Hopf algebras with bijective antipode necessarily surjective? 

\end{question}

And its dual:

\begin{question}

Is a monomorphism of Hopf algebras with bijective antipode necessarily injective? 

\end{question}

These, we believe, should go hand in hand with the aforementioned Kaplansky conjecture and its dual, regarding faithful coflatness.

We now turn our attention to the adjunctions which appear in \seref{3}. It follows immediately from \thref{exadj} (a) that a bialgebra $B$ has a largest subbialgebra which is a quotient of a Hopf algebra (the image of $H_*(B)\to B$), and dually, has a largest quotient bialgebra contained in a Hopf algebra (the image of $B\to H^*(B)$). In an entirely analogous manner, \thref{exadj} (b) implies that a Hopf algebra $H$ has a largest Hopf subalgebra which is a quotient of one with bijective antipode (the image of $K_*(H)\to H$), and a largest quotient Hopf algebra contained in one with bijective antipode (the image of $H\to K^*(H)$). The natural problem arises of characterizing those bialgebras (Hopf algebras) which are quotients or subbialgebras (resp. quotients or Hopf subalgebras) of Hopf algebras (resp. Hopf algebras with bijective antipode). 

For one of the four adjunctions, at least, this question is settled: part of \cite[Prop. 2.7]{Sc} says, in a slightly different formulation, that a Hopf algebra is a Hopf subalgebra of one with bijective antipode iff it has injective antipode. This is a consequence of Schauenburg's construction of $K^*(H)$ as the colimit of the inductive system $u_n:H_n\to H_{n+1},\ n\ge 0$, with $H_n=H$ and $u_n=S_H^2$ for all $n\ge 0$ (see \reref{altadj}). The result just mentioned then follows from the fact that if in such a system all maps are injections, the map sending $H_0$ to the colimit is also an injection. 

As mentioned in \reref{altadj}, we can dualize this construction. The dual statement on inverse limits with surjective maps, however, no longer holds, in general. At least not at the level of coalgebras (and the limit appearing there is one of coalgebras, as the forgetful functor $\halg\to\coalg$ is a right adjoint by \cite{Ta1}, so it preserves limits): one can easily construct a sequence of surjections $C_{n+1}\to C_n$ where $C_n$ are simple coalgebras and $\dim C_n\to\infty$, in which case the resulting limit is none other than $0$. 

Despite such examples, can we still find simple necessary and sufficient conditions on a Hopf algebra in order that it be a quotient of a Hopf algebra with bijective antipode?

\begin{problem}

Characterize those Hopf algebras $H$ for which $K_*(H)\to H$ is surjective. 

\end{problem}

More specifically, we ask

\begin{question}

Is it true that a Hopf algebra with surjective antipode is a quotient of one with bijective antipode? 

\end{question}

And what can be said about the other two adjunctions, between the categories $\bialg$ and $\halg$? We would like to find necessary and sufficient conditions on a bialgebra, expressed intrinsically, in order that it be a subbialgebra or a quotient bialgebra of a Hopf algebra.

\begin{problem}

Characterize intrinsically those bialgebras $B$ for which (a) $B\to H^*(B)$ is injective, or (b) $H_*(B)\to B$ is surjective.   

\end{problem}

We take a moment here to point out that it is by no means true that all bialgebras satisfy (a) (or (b)). In other words, $B\to H^*(B)$ is not always injective, nor is $H_*(B)\to B$ always surjective. Some examples follow.

\begin{example}

Let $M$ be a monoid, and $B=k[M]$ the monoid bialgebra. One sees easily that the free Hopf algebra $H^*(B)$ on $B$ is precisely the group algebra of the enveloping group $G(M)$ of $M$. If the canonical map $M\to G(M)$ happens to be non-injective (and this happens whenever $M$ is not ``cancellable''), $B\to H^*(B)$ will be non-injective as well. This implies that $B$ is not a subbialgebra of a Hopf algebra. 

\end{example}

\begin{example}

Let $H$ be a Hopf algebra with non-injective antipode. It is then clear that $H\to K^*(H)$ cannot be an embedding. In view of \reref{altadj}, $K^*(H)$ is the opposite of $H^*(H^{op})$. Consequently, $B=H^{op}$ is not a subbialgebra of a Hopf algebra. 

\end{example}

\begin{example}

The previous example can be dualized, using \reref{altadj} again: if $H$ is a Hopf algebra with non-surjecive antipode, then $B=H^{op}$ is a bialgebra which is not a quotient of a Hopf algebra. 

\end{example}

\section*{Acknowledgements}

The author would like to thank Prof. Gigel Militaru for suggesting some of the questions posed here, as well as colleagues Ana-Loredana Agore and Drago\c s Fr\u a\c til\u a for countless fruitful conversations on the topics treated in this paper. Also, we thank the referee for valuable suggestions on the improvement of this paper. 



\end{document}